\documentclass[12pt]{amsart}

\usepackage{graphicx} 
\usepackage[utf8]{inputenc}
\usepackage[english]{babel}
\usepackage{amsfonts,amsmath,amssymb,stmaryrd}
\usepackage{faktor}
\usepackage{amsthm}
\usepackage{mathabx}
\usepackage{url}
\usepackage{enumitem}

\usepackage[graphicx]{realboxes}

\usepackage{MnSymbol}
\usepackage{tikz}
\usepackage{tikz-cd}
\usetikzlibrary{calc,shapes.arrows,decorations.pathreplacing}
\usepackage[hyperindex,breaklinks]{hyperref}
\usepackage{cleveref}
\newcommand{\indep}[4]{#1\underset{#2}{\downfree^{#4}}#3}
\newcommand{\notIndep}[4]{#1\underset{#2}{{\not\downfree}^{#4}}#3}

\theoremstyle{plain}

\newtheorem{theorem}{Theorem}[section]
\newtheorem{proposition}[theorem]{Proposition}
\newtheorem{fact}[theorem]{Fact}
\newtheorem{lemma}[theorem]{Lemma}
\newtheorem{corollary}[theorem]{Corollary}

\theoremstyle{remark}
\newtheorem{remark}[theorem]{Remark}
\theoremstyle{definition}
\newtheorem{definition}[theorem]{Definition}
\newcommand{\tp}{\textup{tp}}

\newcommand{\acl}{\textup{acl}}
\newcommand{\dcl}{\textup{dcl}}
\renewcommand{\subset}{\subseteq}
\renewcommand{\supset}{\supseteq}
\renewcommand{\epsilon}{\varepsilon}

\newcommand{\Aut}{\textup{Aut}}

\newcommand{\Ker}{\textup{Ker}}
\renewcommand{\Im}{\textup{Im}}
\newcommand{\RV}{\textup{RV}}

\renewcommand{\d}{\mathbf{d}}
\newcommand{\f}{\mathbf{f}}

\begin{document}
\author{Akash HOSSAIN}
\thanks{Part of the research was carried out during a 2 months visit of the
author to the University of Münster, which was supported by the German
Research Foundation (DFG) via HI 2004/1-1 (part of the French-German
ANR-DFG project GeoMod).}
\title[Forking and dividing in expansions of PSES]{Transfer principles for forking and dividing in expansions of pure short exact sequences of Abelian groups}
\date{}
\begin{abstract}
In their article about distality in valued fields, Aschenbrenner, Chernikov, Gehret and Ziegler proved resplendent Ax-Kochen-Ershov principles for quantifier elimination in pure short exact sequences of Abelian structures. We study how their work relates to forking, and we prove Ax-Kochen-Ershov principles for forking and dividing in this setting.
\end{abstract}
\maketitle

\section{Introduction}
Suppose we have a family $(M_i)_i$ of first-order structures, the model theory of which is fairly well-understood, and we use the $M_i$ to build a common expansion $M$ (the model theory of which we want to understand). In this setting, when one wants to understand a model-theoretic notion in $M$ (in this paper, we are interested in forking), it makes sense to try to understand how this property behaves in the $M_i$, and what is the relation between its behavior in the $M_i$ and in $M$. This question is of even greater interest when the $M_i$ are stably-embedded in $M$, and when they are orthogonal.
\par The model theory of \textit{Henselian valued fields} is a field of study where a great volume of work has been done in this regard. In this setting, one can often reduce to working in the model theory of the value group $\Gamma$ and the residue field $k$ of the valued field at hand, thanks to what we call Ax-Kochen-Ershov (AKE)-principles. Such transfer principles are not always possible, and when they are not, the obstruction often comes from \textit{immediate extensions}, which cannot be fully-described in terms of their reductions to $k$ and $\Gamma$. There is a reduct which, in some sense, encodes all the data about the extensions where there is no such obstruction: the \textit{leading-term structure} $\RV$ (see \cite{BASARAB199151}, \cite{flenner}), which sits inside the short exact sequence of Abelian groups:
$$
1\longrightarrow k^*\longrightarrow\RV\longrightarrow\Gamma\longrightarrow 0
$$
In this nicer structure, AKE principles are easier to obtain, and they often suffice to prove strong results about the whole valued field (see for instance \cite{suitesExactes}, Section 5). This particular example motivates the study of the (more abstract) model theory of short exact sequences of Abelian groups. The above short exact sequence is not taken as it is, we expand it by the ordered group structure on $\Gamma$, and the ring structure on $k$. Moreover, reductions to $\RV$ can also be done in a more \textit{resplendent} setting, where one looks at an expansion of a Henselian valued field, and the induced expansion of its leading-term structure. This is applied in particular to the model theory of differential/difference valued fields (see for instance \cite{suitesExactes}, Section 7).
\par Therefore, the question we ask ourselves is: given a short exact sequence of Abelian groups:
$$
0\longrightarrow A\longrightarrow B\longrightarrow C\longrightarrow 0
$$
with an arbitrary expansion on $A$ and $C$, when can a model-theoretic problem on the whole structure be reduced to the reducts $A$ and $C$?
\par Our main reference is the work of Aschenbrenner-Chernikov-Gehret-Ziegler (\cite{suitesExactes}, Section 4), where they obtain such transfer principles for quantifier elimination and distality, and variants of these principles in a more general setting involving Abelian structures.
\par In this paper, we work in the same general setting, and we use their results to establish transfer principles for forking and dividing. Given $M$ a pure short exact sequence of Abelian structures:
$$
0\longrightarrow A\longrightarrow B\longrightarrow C\longrightarrow 0
$$
with arbitrary expansions on $A$ and on $C$, given substructures $\Gamma\subset\Delta\subset M^{eq}$, and $u$ a tuple from $M$, we work out when $u$ is independent from $\Delta$ over $\Gamma$ with respect to forking and dividing ($\indep{u}{\Gamma}{\Delta}{\f}$ and $\indep{u}{\Gamma}{\Delta}{\d}$), in terms of what happens in $A$ and $C$. In \cref{defACtypes}, we isolate the notion of $A$\textit{-types} and $C$\textit{-types}, two partial types $\tp_A(u/\Delta)$ and $\tp_C(u/\Delta)$ realized by $u$, which encode the relations between $u$ and $\Delta$ taking place in $A$ and $C$. We consider in \cref{defTypesPartiels} a partial type $p_0$ over $\Gamma$, which is satisfied by $u$. Our main result is as follows:
\begin{theorem}[see \cref{thmDeviation}]
Suppose $\Gamma$ and $\Delta$ satisfy the hypothesis \ref{GammaAlg} and \ref{DeltaReel} from \cpageref{GammaAlg}. Then we have $\indep{u}{\Gamma}{\Delta}{\f}$ if and only if both partial types $\tp_C(u/\Delta)$ and $p_0\cup \tp_A(u/\Delta)$ do not fork over $\Gamma$.
\par Likewise (see the end of \cref{PSESSectDiv}), we have $\indep{u}{\Gamma}{\Delta}{\d}$ if and only if for each of those two partial types, there exists some realization $v$ such that $\indep{v}{\Gamma}{\Delta}{\d}$.
\end{theorem}
Moreover, we give in \cref{PSESDivisionVraiHomeo} and \cref{PSESDeviationVraiHomeo} descriptions of the spaces of types over $\Delta$ extending $p_0$ which do not fork/divide over $\Gamma$, in terms of the $A$-types and the $C$-types.
\par In \cref{PSESExemple}, we give an example which illustrates why $p_0$ is necessary in our main theorem. This partial type $p_0$ essentially encodes \textit{non-lifted torsors}, cosets in (some quotient of) $B$ of (some quotient of) $A$, from which we do not have a bijection into $A$ which is definable with the parameter sets at hand. The obstruction to AKE principles comes from the fact that we cannot treat those definable sets as copies of (some quotient of) $A$.
\par Finally, in \cref{sectPSESModeles}, we make additional hypothesis on the parameter set $\Gamma$ (see \cref{defPSESModeles}), which are in particular satisfied when $\Gamma$ is a model. With those hypothesis, we have a much stronger AKE principle:
\begin{theorem}[see \cref{PSESModelesDivisionFinal} and \cref{PSESModelesDeviation}]
Let $U$ be the substructure generated by $u$ and by $A^*(\Gamma)\cup B(\Gamma)\cup C^*(\Gamma)$. Then we have $\indep{U}{\Gamma}{\Delta}{\d}$ if and only if $\indep{A^*(U)}{A^*(\Gamma)}{A^*(\Delta)}{\d}$ and $\indep{C^*(U)}{C^*(\Gamma)}{C^*(\Delta)}{\d}$ in the respective reducts $A^*(M)$ and $C^*(M)$.
\par Likewise, we have $\indep{U}{\Gamma}{\Delta}{\f}$ if and only if we have $\indep{A^*(U)}{A^*(\Gamma)}{A^*(\Delta)}{\f}$ and $\indep{C^*(U)}{C^*(\Gamma)}{C^*(\Delta)}{\f}$ in the respective reducts $A^*(M)$ and $C^*(M)$.
\end{theorem}

Let us set up some notations. Just as in \cite{suitesExactes}, we denote by $A$, $B$, $C$ the terms of the short exact sequences at hand, therefore they will refer to multi-sorted structures that are reducts of the short exact sequence. Parameter sets will be denoted by $\Gamma$, $\Delta$, $U$. Our tuples will always be tuples of reals, and they will be denoted by $u$, $v$, $w$... The real sorts will be in the disjoint reunion of $A^*$, $B$, $C^*$, with $A^*$ an expansion of $A$, and $C^*$ an expansion of $C$. As a result, a tuple $u$ will be the concatenation of three disjoint subtuples $u_A$, $u_B$, $u_C$, with of course $u_A\in (A^*)^{<\omega}$, $u_B\in B^{<\omega}$, $u_C\in (C^*)^{<\omega}$. We shall always use those subscripts in that fashion, and we will not redefine what they refer to.
\section{Independence}
The non-forking relation is an independence notions introduced by Shelah \cite{Shelah1982-SHECTA-5} during the developments of stability theory. It plays a fundamental role in several classes of theories, notably stable, simple and NTP$_2$ theories.
\par Without going into details about its conceptual role in model theory, we briefly recall basic definitions in this section. We refer the reader to \cite{Tent2012ACI} for a more detailed study of forking and dividing. Let $M$ be a first-order structure, and let $\Gamma$, $\Delta$ be parameter subsets of $M$.
\begin{definition}
Let $\phi(x, \delta)$ be some formula with parameters $\delta$. Then $\phi(x, \delta)$ \textit{divides over $\Gamma$}, when there exists some $\Gamma$-indiscernible sequence $(\delta_i)_i$ starting at $\delta$ such that the partial type:
$$
\left\lbrace\phi(x, \delta_i)|i\right\rbrace
$$
is inconsistent.
\par We say that $\phi(x, \delta)$ \textit{forks over $\Gamma$} if, in some elementary extension, there exists a tuple $\epsilon$, and finitely-many formulas $(\psi_i(x, \epsilon))_i$ with parameters $\epsilon$, such that each $\psi_i(x, \epsilon)$ divides over $\Gamma$, and:
$$
\phi(x, \delta)\models\bigvee\limits_i\psi_i(x, \epsilon)
$$
We say that a partial type $\pi$ \textit{forks/divides over $\Gamma$} if $\pi$ implies a formula which forks/divides over $\Gamma$.
\end{definition}

\begin{proposition}\label{divisionPreimage}
If $f$ is a $\Gamma$-definable function, and $Y$ is a definable subset of the image of $f$, then $Y$ divides over $\Gamma$ if and only if $f^{-1}(Y)$ does.
\end{proposition}

\begin{proof}
Write $Y=\phi(y, \delta)$. We may assume $M$ is $|\Gamma|^+$-saturated and strongly-$|\Gamma|^+$-homogeneous. Let $(\sigma_n)_n\in \Aut({M}/\Gamma)^\omega$ be such that the sequence $(\sigma_n(\delta))_n$ is $\Gamma$-indiscernible. As $f$ is $\Gamma$-definable, we have the equality:
$$\sigma_n(f^{-1}(Y))=f^{-1}(\sigma_n(Y))$$
Now:
$$\bigcap\limits_{n}\sigma_n(f^{-1}(Y))=\bigcap\limits_{n}f^{-1}(\sigma_n(Y))=f^{-1}\left(\bigcap\limits_{n}\sigma_n(Y)\right)$$
As $Y$ and its $\Gamma$-conjugates are subsets of the image of $f$, such an intersection is empty if and only if $\bigcap\limits_{n}\sigma_n(Y)=\emptyset$. As a result, $(\sigma_n(\delta))_n$ is a witness for division of $Y=\phi(y, \delta)$ over $\Gamma$ if and only it is also a witness for division of $f^{-1}(Y)=\phi(f(x), \delta)$ over $\Gamma$, which concludes the proof.
\end{proof}

\begin{fact}\label{indepDivisionSuitesIndisc}
The following are equivalent:
\begin{itemize}
\item $\tp(u/\Gamma\Delta)$ does not divide over $\Gamma$.
\item For every $\Gamma$-indiscernible sequence $I$ starting at (some enumeration of) $\Delta$, there exists some $J\equiv_{\Gamma\Delta} I$ such that $J$ is $\Gamma u$-indiscernible.
\end{itemize}
\end{fact}

\begin{definition}
We write $\indep{u}{\Gamma}{\Delta}{\d}$ whenever the above conditions hold.
\end{definition}

\begin{fact}\label{deviationExtGlobale}
The following are equivalent:
\begin{itemize}
\item $\tp(u/\Gamma\Delta)$ does not fork over $\Gamma$.
\item For some/any $|\Gamma|^+$-saturated model $N$ containing $\Gamma\Delta$, there exists in some elementary extension of $N$ some $v\equiv_{\Gamma\Delta} u$ such that $\indep{v}{\Gamma}{N}{\d}$.
\end{itemize}
\end{fact}

\begin{definition}
We write $\indep{u}{\Gamma}{\Delta}{\f}$ whenever the above conditions hold.
\end{definition}

\begin{remark}
If $\Delta$ is a $|\Gamma|^+$-saturated model containing $\Gamma$, then we have:
$$
\indep{u}{\Gamma}{\Delta}{\d}\Longleftrightarrow \indep{u}{\Gamma}{\Delta}{\f}
$$
\end{remark}

\section{Preliminaries and resplendent quantifier elimination}

In this section, we define what are pure short exact sequences of Abelian structures, and we recall results from previous literature.

\begin{definition}\label{PSESDefPP}
Fix $\mathcal{L}$ a first-order language. We define the set of \textit{p.p.-formulas} (for Primitive Positive) as the closure of the set of parameter-free atomic formulas under finite conjunction and existential quantification.
\par Let $A$ be a multi-sorted first-order structure with language $\mathcal{L}$. A \textit{fundamental system} for $A$ is a set $\Phi$ of p.p.-formulas such that, modulo the theory of $A$, every p.p.-formula is equivalent to a finite conjunction of formulas of the form $\phi(t(x))$, with $\phi\in\Phi$, and $t$ a tuple of parameter-free terms.
\par We define $A$ to be an \textit{Abelian structure} when the following conditions hold:
\begin{itemize}
\item On each sort $s$, the structure $A$ expands a structure of Abelian group on $s(A)$ .
\item Any parameter-free term $t(x_{s_1}\ldots x_{s_n})$ (seen as a $\emptyset$-definable function $s_1(A)\times\ldots \times s_n(A)\longrightarrow s(A)$) is a group homomorphism.
\item For every p.p.-formula $\phi(x_{s_1}\ldots x_{s_n})$, the set $\phi(A)$ is a subgroup of $s_1(A)\times\ldots \times s_n(A)$.
\end{itemize}
\par Given Abelian structures $A$, $B$ on the language $\mathcal{L}$, an embedding (for the language $\mathcal{L}$) $\iota:\ A\longrightarrow B$ is said to be \textit{pure} when we have $\iota^{-1}(\phi(B))=\phi(A)$ for every p.p.-formula $\phi$.
\par Let $\Phi$ be a set of p.p.-formulas. A \textit{pure short exact sequence of Abelian structures expanded by $\Phi$} ($\Phi$-PSES) is a first-order structure
\begin{center}
\begin{tikzcd}
    A\arrow[rr, "\iota"]\arrow[dr, "(\pi_\phi)_{\phi\in\Phi}"]&&B\arrow[rr, "\nu"]\arrow[dl, "(\rho_\phi)_{\phi\in\Phi}"]&&C\\
    &\left(\faktor{A}{\phi(A)}\right)_{\phi\in\Phi}&&&
\end{tikzcd}
\end{center}
consisting of:
\begin{itemize}
\item Three Abelian structures $A$, $B$, $C$ on copies of $\mathcal{L}$ such that the respective sets of sorts of $A$, $B$, $C$ are pairwise-disjoint, together with the quotient groups $\faktor{A}{\phi(A)}$ for each $\phi\in\Phi$.
\item A pure embedding $\iota:\ A\longrightarrow B$, and a surjective morphism $\nu:\ B\longrightarrow C$ such that, on each sort of $\mathcal{L}$, we have $\Im(\iota)=\Ker(\nu)$.
\item For each $\phi\in\Phi$, the quotient map $\pi_\phi:\ A\longrightarrow \faktor{A}{\phi(A)}$.
\item For each $\phi\in\Phi$, the map $\rho_\phi:\ B\longrightarrow \faktor{A}{\phi(A)}$, which is zero outside of $\phi(B)+\iota(A)$, and extends the group homomorphism:
$$
\phi(B)+\iota(A)\longrightarrow\faktor{(\phi(B)+\iota(A))}{\phi(B)}\simeq\faktor{A}{\phi(A)}
$$
\end{itemize}
\end{definition}
Note that it follows from the axioms of PSES that we have $$\nu^{-1}(\phi(C))=\phi(B)+\iota(A)$$ for every p.p.-formula $\phi$. Equivalently, $\nu$ is a \textit{pure projection}, that is for every p.p.-formula $\phi$, we have $\phi(C)=\nu(\phi(B))$.
\par For the remainder of this paper, let $M^-$ be a $\Phi$-PSES:
$$
0\longrightarrow A\underset{\iota}{\longrightarrow} B\underset{\nu}{\longrightarrow} C\longrightarrow 0
$$
with $\Phi$ a fundamental system for $B$. Let $\mathcal{L}_B$ be the language of $B$, and let $\mathcal{L}_A$, $\mathcal{L}_C$ be arbitrary disjoint enrichments of the respective languages of $A$ and $C$, such that the sort $\faktor{A}{\phi(A)}$ and the map $\pi_\phi$ are in $\mathcal{L}_A$ for every $\phi\in\Phi$. Let $A^*$ be an expansion of $A$ with language $\mathcal{L}_A$, and let $C^*$ be an expansion of $C$ with language $\mathcal{L}_C$. Note that $A^*$, $C^*$ are completely arbitrary expansions of $A$ and $C$, they are in general not interpretable. Let $M$ be the expansion

$$
0\longrightarrow A^*\underset{\iota}{\longrightarrow} B\underset{\nu}{\longrightarrow} C^*\longrightarrow 0
$$

\noindent of $M^-$ obtained by expanding $A$, $C$ to $A^*$, $C^*$. Our goal is to obtain a characterization for dividing and forking in $M$ relative to dividing and forking in $A^*(M)$ and $C^*(M)$.
\par As mentioned in the introduction, a specific case of this very general context is important to understand (expansions of) valued fields. One can often reduce a problem taking place in a valued field $K$ to a problem on the short exact sequence:
$$
0\longrightarrow k\longrightarrow \RV^*\longrightarrow\Gamma\longrightarrow 0
$$ with $k$ the multiplicative group of the residue field of $K$ expanded to the field structure $k$, $\mathfrak{M}$ the maximal ideal of the valuation ring of $K$, $\RV^*=\faktor{K^*}{1+\mathfrak{M}}$ the group of non-zero leading terms of $K$, and $\Gamma$ the value group of $K$ expanded to its ordered group structure and the value of zero. In that case, the Abelian structures at hand are merely Abelian groups, and the fundamental system that we choose for $B$ is usually the set of formulas of the form $\exists y\ n\cdot y=x$, with $n$ a natural integer (including $0$). The notion of pure embedding that we defined corresponds in this context to the usual notion of pure embeddings of Abelian groups. Reductions to $\RV$ also exist in common expansions of valued fields, such as differential valued fields, ordered valued fields or difference valued fields, and they 
may correspond to richer expansions of $k$ and $\Gamma$.
\par Now we describe known quantifier elimination results from previous literature.
\begin{definition}\label{defACtypes}
\par Let $\Gamma\subset M$ be a substructure. We  see $A^*$, $B$, $C^*$ as (tuples of) sorts, and we recall that $A^*(\Gamma)$, $B(\Gamma)$, $C^*(\Gamma)$ refer to the sets of elements of $\Gamma$ that are in $A^*$, $B$, $C^*$.
\par A $C$\textit{-formula} with parameters in $\Gamma$ is a formula of the form: 
$$\psi_C(\nu(x_B), x_C, \gamma)$$
with $\psi_C\in\mathcal{L}_C$ parameter-free, and $\gamma\in C^*(\Gamma)$. If $u$ is a tuple from $M$, then we write $\tp_C(u/\Gamma)$ for the $C$-type of $u$ over $\Gamma$, that is the partial type of every $C$-formula with parameters in $\Gamma$ satisfied by $u$.
\par An $A$\textit{-formula} with parameters in $\Gamma$ is a formula of the form:
$$
\psi_A(\alpha, x_A, (\rho_{\phi_i}(t_i(x_B)-\beta_i))_i)
$$
with $\psi_A\in\mathcal{L}_A$ parameter-free, $(\phi_i)_i$ a finite family of elements of $\Phi$, $(t_i)_i$ a finite family of parameter-free $\mathcal{L}_B$-terms, $\alpha\in A^*(\Gamma)$, and $\beta_i\in B(\Gamma)$. We define $\tp_A(u/\Gamma)$ similarly.
\end{definition}
Note that, given $t(x_1\ldots x_n, y_1\ldots y_m)$ a parameter-free $\mathcal{L}_B$-term, and given $\beta_1\ldots \beta_m\in B(\Gamma)$, $t(x_1\ldots x_n, \beta_1\ldots \beta_m)$ is equivalent to the term:
$$t(x_1\ldots x_n, 0\ldots 0)-t(0\ldots 0, -\beta_1\ldots -\beta_m)$$
thus it may be written just as in the definition of $A$-formulas, as a term of the form $s(x_1\ldots x_n)-\beta$, with $s$ parameter-free, and $\beta\in B(\Gamma)$.
\par We will be using the following quantifier elimination result:
\begin{theorem}[\cite{suitesExactes}, Corollary 4.20]\label{qe}
If $\Gamma$ is generated by reals (i.e. it is a substructure generated by actual points from $A^*(M)$, $B(M)$ and $C^*(M)$), then we have:
$$\tp(u/\Gamma)=\tp(v/\Gamma)\Longleftrightarrow \left\lbrace
\begin{array}{c}
\tp_A(u/\Gamma)=\tp_A(v/\Gamma)\\
\tp_C(u/\Gamma)=\tp_C(v/\Gamma)\end{array}\right.$$
\end{theorem}
We specify that $\Gamma$ is generated by reals because, in the next sections, we work with imaginary parameter sets.
\par In fact, this result is proved in \cite{suitesExactes} under weaker assumptions: the expansion is allowed to contain functions and relations that involve both $A$ and $C$. However, under these assumptions, $A^*$ and $C^*$ may not be orthogonal, making it unreasonable (and arguably impossible at this level of generality) to try to obtain transfer principles for forking. For the same reason, we do not work in the setting of weakly pure short exact sequences of Abelian structures (see \cite{suitesExactes}, subsubsection 4.5.2), as $\nu\circ\iota$ would be a non-trivial map from $A^*$ to $C^*$. In our setting, $A^*$ and $C^*$ are orthogonal, and they are stably embedded in a strong sense:
\begin{corollary}\label{stablePlongitude}
If $\Gamma$ is generated by reals, then we have $u_Au_C\equiv_\Gamma v_Av_C$ in $M$ if and only if $u_A\equiv_{A^*(\Gamma)} v_A$ in the reduct $A^*(M)$, and $u_C\equiv_{C^*(\Gamma)} v_C$ in $C^*(M)$. In particular, $u_A\equiv_\Gamma v_A$ in $M$ if and only if $u_A\equiv_{A^*(\Gamma)} v_A$ in $A^*(M)$, and likewise for $C$.
\end{corollary}

\begin{remark}\label{PSESQERque}
\par We should point out that $A$-types and $C$-types are not as independent from each other as one would expect, due to the technical definition of $A$-formulas. For instance, the $A$-formula $\rho_\phi(x)\neq 0$ must imply the $C$-formula $\phi(\nu(x))$. Because of that, one has to be careful when using orthogonality to study forking in this setting. We build in \cref{PSESExemple} a type $\tp(u/\Delta)$ over a larger substructure $\Delta\supset\Gamma$ which divides over $\Gamma$, while $\tp_A(u/\Delta)$ and $\tp_C(u/\Delta)$ do not fork over $\Gamma$. 
\par However, even though naive transfer principles fail for forking and dividing in the full Stone space of types over $\Delta$, they hold in a subspace which only depends on the base parameter set $\Gamma$. More precisely, given $u$ and $\Delta$, we will define $p_0$ a partial type over $\Gamma$ realized by $u$ such that, in the Stone space of types over $\Delta$ containing $p_0$, $A$-types and $C$-types \textit{are} independent from each other. We do manage to classify, in an Ax-Kochen-Ershov fashion, the types in this Stone space which do not fork/divide over $\Gamma$, in terms of the induced $A$-types and $C$-types. This yields in particular a classification of the space of non-forking extensions of $\tp(u/\Gamma)$.
\end{remark}

\section{Orthogonality}
Just as in the previous section, we have $M$ an expansion of a $\Phi$-PSES:
$$
0\longrightarrow A^*\underset{\iota}{\longrightarrow} B\underset{\nu}{\longrightarrow} C^*\longrightarrow 0
$$
with $\Phi$ a fundamental system for $B$, we have substructures $\Gamma\subset\Delta\subset M^{eq}$ (we allow imaginaries this time), and $u=u_Au_Bu_C\in M$ a tuple. We defined $A$-types and $C$-types in the last section, and we noted that they are not completely independent from each other. We get around that problem in this section.
\par As we want to describe forking, we may assume that $M$ is $|\Delta|^+$-saturated and strongly $|\Delta|^+$-homogeneous. Likewise, we may freely assume that $\Gamma=\dcl^{eq}(\Gamma)$, $\Delta=\dcl^{eq}(\Delta)$. For technical reasons, we also make the following two mild hypothesis:
\begin{enumerate}[label=(H\arabic*), ref=(H\arabic*)]
\item\label{GammaAlg} We assume that $\Gamma$ has enough algebraic imaginaries: we assume that for every $\phi\in\Phi$, each coset from $\faktor{C}{\phi(C)}$ which belongs to $\acl^{eq}(\Gamma)$ belongs in fact to $\Gamma$.
\item\label{DeltaReel} We assume that $\Delta$ is generated by reals:
$$\Delta=\dcl^{eq}(A^*(\Delta)\cup B(\Delta)\cup C^*(\Delta))$$
\end{enumerate}
%

\noindent\textbf{Notations} Our goal is to describe the Stone space of extensions in $S(\Delta)$ of $\tp(u/\Gamma)$ which do not fork/divide over $\Gamma$. To achieve this, we define various partial types, Stone spaces and definable functions. The partial types that we define typically state that $\nu(x_B)$ does or does not belong to various cosets of $\faktor{C}{\phi(C)}$. One partial type, $p_0$, depends solely on $\tp(u/\Gamma)$, and plays an important role in our main results. We write with the subscript $_0$ the various objects (such as Stone spaces) which depend on $p_0$ in some way. Other objects will depend on the parameter set $\Delta$, we will typically write them with the subscript/superscript $_\Delta$.

\begin{definition}\label{defTypesPartiels}
Let $\mathcal{T}$ be the set of parameter-free $\mathcal{L}_B$-terms. We define the following sets:
$$
E^+_0=\left\lbrace (\phi, t)\in \Phi\times\mathcal{T}|\nu(t(u_B))+\phi(C)\in \faktor{C}{\phi(C)}(\Gamma)\right\rbrace
$$
$$
E^-_0=(\Phi\times\mathcal{T})\setminus E^+_0
$$
$$
E^+_\Delta=\left\lbrace (\phi, t)\in E^+_0|\exists \beta\in B(\Delta)\ \nu(t(u_B)-\beta)\in\phi(C)\right\rbrace
$$
For each $(\phi, t)\in E^+_\Delta$, choose $\beta_{(\phi, t)}\in B(\Delta)$ a witness of the last equation.
\par By definition of $E^+_0$ and $E^-_0$, the following partial type:
$$p_0(x_A, x_B, x_C)=\left\lbrace \nu(t(x_B-u_B))\in \phi(C))|(\phi, t)\in E^+_0
\right\rbrace
$$
$$\cup\left\lbrace \nu(t(x_B))+\phi(C)\neq X|(\phi, t)\in E^-_0, X\in\faktor{C}{\phi(C)}(\Gamma)
\right\rbrace$$
can be written with formulas over $\Gamma$. We also define $p_\Delta(x_A, x_B, x_C)$ as the following partial type over $\Delta$:
$$
p_0(x_A, x_B, x_C)\cup\left\lbrace \neg\phi(\nu(t(x_B)-\beta))|(\phi, t)\in E^-_0, \beta\in B(\Delta)
\right\rbrace$$
\par Finally, we define the $\emptyset$-definable function:
$$f_C:\ (x_A, x_B, x_C)\longmapsto (\nu(x_B), x_C)$$
and we define the (infinite) tuple of $\Delta$-definable functions:
$$f_A^\Delta:\ (x_A, x_B, x_C)\longmapsto (x_A, (\rho_\phi(t(x_B)-\beta_{(\phi, t)}))_{(\phi, t)\in E^+_\Delta})$$
The graph of $f_A^\Delta$ is $\Delta$-$*$-definable.
\end{definition}
Our goal in this paper is to classify the types over $\Delta$ which contain $p_0$ and do not fork/divide over $\Gamma$. We see $f_C$ and $f_A^\Delta$ as projections to $A^*$ and $C^*$. With respect to those projections, the sets of realisations of $p_0$ and $p_\Delta$ are preimages by $f_C$, i.e. if $u\models p_0$ (resp. $u\models p_\Delta$), and $f_C(u)=f_C(v)$, then it follows from the definitions that $v\models p_0$ (resp. $v\models p_\Delta$).
\par We call $f_A^\Delta$ and $f_C$ projections, but $f_A^\Delta\times f_C$ may not be onto the full direct product, i.e. its direct image may not be a rectangle, as \cref{PSESQERque} points out. The following statements shows that the direct images of $p_0$ and $p_\Delta$ by $f_A^\Delta\times f_C$ are, in fact, rectangles:

\begin{proposition}\label{surjection}
Let $v, v'\in M$ be realizations of $p_0$. Then there exists $w\in M$ such that $f_A^\Delta(v)=f_A^\Delta(w)$, and $f_C(v')=f_C(w)$ (in particular, $w\models p_0$).
\end{proposition}
\begin{proof}
Let $(\phi_i, t_i)_i$ be a finite family from $E^+_\Delta$. We  find $w\in M$ such that $f_C(w)=f_C(v')$, and $\rho_{\phi_i}(t_i(w_B)-\beta_{(\phi_i, t_i)})=\rho_{\phi_i}(t_i(v_B)-\beta_{(\phi_i, t_i)})$ for every $i$, and the statement will follow by saturation.
\par Let $\psi(x)$ be the following p.p.-formula:
$$
\bigwedge\limits_i\phi_i(t_i(x))
$$
then the map $x_B\longmapsto (\rho_{\phi_i}(t_i(x_B)-\beta_{(\phi_i, t_i)}))_i$ factors through the quotient $\faktor{B}{\psi(B)}$, thus it suffices to find $w_B$ such that $\nu(w_B)=\nu(v'_B)$ (in other words, $w_B-v'_B\in\iota(A)$) and $w_B-v_B\in\psi(B)$, and set $w=(v_A, w_B, v'_C)$. Such a $w_B$ exists if and only if $v_B-v'_B\in\psi(B)+\iota(A)$, if and only if $\nu(v_B-v'_B)\in\psi(C)$, if and only if $\nu(t_i(v_B-v'_B))\in \phi_i(C)$ for each $i$. Now, $v$ and $v'$ both realize $p_0$, and we have as a result $\nu(t_i(u_B-v_B))\in\phi_i(C)\ni\nu(t_i(u_B-v'_B))$, concluding the proof.
\end{proof}
As the set of realizations of $p_\Delta$ is contained in that of $p_0$, and is also a preimage by $f_C$, it immediately follows that:
\begin{corollary}\label{corSurjection}
Proposition \ref{surjection} also holds if we replace $p_0$ by $p_\Delta$.
\end{corollary}

Note that \cref{surjection} is not stated in its most optimal form: it holds for the following weaker partial type:
$$
\left\lbrace \nu(t(x_B-u_B))\in \phi(C))|(\phi, t)\in E^+_0
\right\rbrace
$$

\begin{remark}
Let $S_0^\Delta$ (resp. $S^\Delta$) be the Stone space of every type from $S^{u}(\Delta)$ which extends $p_0$ (resp. $p_\Delta$). Then $(f_A^\Delta, f_C)$ yield continuous maps $S_0^\Delta\longrightarrow f_A^\Delta(S_0^\Delta)\times f_C(S_0^\Delta)$, and $S^\Delta\longrightarrow f_A^\Delta(S^\Delta)\times f_C(S^\Delta)$. By \cref{stablePlongitude} and \ref{DeltaReel}, the factors of those direct products can be identified with closed subspaces of $S(A^*(\Delta))$ and $S(C^*(\Delta))$ in the first-order structures $A^*(M)$ and $C^*(M)$. Then, \cref{surjection} and \cref{corSurjection} state that those two continuous maps are \textit{surjective}.
\end{remark}

\begin{proposition}\label{homeoRectangle}
The map $S^\Delta\longrightarrow f_A^\Delta(S^\Delta)\times f_C(S^\Delta)$ is a homeomorphism.
\end{proposition}
\begin{proof}
By compactness of the domain, separation of the image, and surjectivity, we just need to show injectivity. Let $v$, $w$ be realizations of $p_\Delta$ such that $f_A^\Delta(v)\equiv_{A^*(\Delta)}f_A^\Delta(w)$ and $f_C(v)\equiv_{C^*(\Delta)}f_C(w)$. We need to show that $v\equiv_\Delta w$. By \cref{stablePlongitude} and \ref{DeltaReel}, we have $f_A^\Delta(v)f_C(v)\equiv_\Delta f_A^\Delta(w)f_C(w)$. As a result, by strong homogeneity, we can freely assume that $f_A^\Delta(v)=f_A^\Delta(w)$ and $f_C(v)=f_C(w)$.
\par Now let us show $v\equiv_\Delta w$. Let $\Delta_r=A^*(\Delta)\cup B(\Delta)\cup C^*(\Delta)$. By \ref{DeltaReel}, it suffices to show that $v\equiv_{\Delta_r} w$. By \cref{qe}, it suffices to show that $\tp_A(v/\Delta_r)=\tp_A(w/\Delta_r)$ and $\tp_C(v/\Delta_r)=\tp_C(w/\Delta_r)$.
\par The fact that $\tp_C(v/\Delta_r)=\tp_C(w/\Delta_r)$ immediately follows from the fact that $f_C(v)=f_C(w)$.
\par Let us show $\tp_A(v/\Delta_r)=\tp_A(w/\Delta_r)$. Let $t$ in $\mathcal{T}$, $\phi$ in $ \Phi$, and $\beta$ in $ B(\Delta_r)=B(\Delta)$. It suffices to show that $\rho_\phi(t(v_B)-\beta)=\rho_\phi(t(w_B)-\beta)$.
\begin{itemize}
\item Suppose $(\phi, t)\in E^+_\Delta$, and $\beta-\beta_{(\phi, t)}\in \phi(B)+\iota(A)$. Then:
$$
\begin{array}{ccl}
\rho_\phi(t(v_B)-\beta) &= &\rho_\phi(t(v_B)-\beta_{(\phi, t)})+\rho_\phi(\beta_{(\phi, t)}-\beta)\\
 &= &\rho_\phi(t(w_B)-\beta_{(\phi, t)})+\rho_\phi(\beta_{(\phi, t)}-\beta)\\
 &=&\rho_\phi(t(w_B)-\beta)
\end{array}
$$
\par The second equality follows from the fact that $f_A^\Delta(v)=f_A^\Delta(w)$, the other equalities follow from the definition of $\rho_\phi$.
\item Suppose $(\phi, t)\in E^+_\Delta$, and $\beta-\beta_{(\phi, t)}\not\in \phi(B)+\iota(A)$. Then:
$$t(v_B)-\beta\not\in\phi(B)+\iota(A)\not\ni t(w_B)-\beta$$
thus $\rho_\phi(t(v_B)-\beta)=0=\rho_\phi(t(w_B)-\beta)$.
\item Suppose $(\phi, t)\in E^+_0\setminus E^+_\Delta$. By definition of $p_0$, we have:
$$t(v_B-u_B)\in \phi(B)+\iota(A)\ni t(w_B-u_B)$$ however, by definition of $E^+_\Delta$, we have $t(u_B)-\beta\not\in\phi(B)+\iota(A)$, thus we conclude just as in the above item that $\rho_\phi(t(v_B)-\beta)=0=\rho_\phi(t(w_B)-\beta)$.
\item Suppose $(\phi, t)\in E^-_0$. Then, by definition of $p_\Delta$, we have:
$$t(v_B)-\beta\not\in\phi(B)+\iota(A)\not\ni t(w_B)-\beta$$
just as in the above item.
\end{itemize}
This concludes the proof, since all cases were considered.
\end{proof}

\begin{remark}\label{orthACDelta}
Let $v$, $w$ be realizations of $p_\Delta$. By definition, $\tp_C(v/\Delta)=\tp_C(w/\Delta)$ if and only if $f_C(v)\equiv_\Delta f_C(w)$, and the above proof shows us that the (non-trivial) same equivalence holds for $A$: $\tp_A(v/\Delta)=\tp_A(w/\Delta)$ if and only if $f_A^\Delta(v)\equiv_\Delta f_A^\Delta(w)$. So we have a homeomorphism between the space of $A$-types of realizations of $p_\Delta$, and $f_A^\Delta(S^\Delta)$. In particular, $S^\Delta$ is homeomorphic to the direct product of the respective spaces of $A$-types and $C$-types of realizations of $p_\Delta$. 
%
%
%
%

\end{remark}

\section{Dividing}\label{PSESSectDiv}
We defined in the previous section a space of types $S_0^\Delta$ which is closed in $S(\Delta)$. We  classify here the subspace of $S_0^\Delta$ of types which do not divide over $\Gamma$ in terms of the projection maps $f_A^\Delta$, $f_C$. With respect to those maps, we saw that $S_0^\Delta$ maps onto a rectangle, while $S^\Delta$ \textit{is} a rectangle. Thus the elements of $S^\Delta$ are in some sense easier to classify than in $S_0^\Delta$. Thankfully, the space that we are trying to describe is a subspace of $S^\Delta$:

\begin{proposition}\label{divisionPDelta}
Every type in $S_0^\Delta\setminus S^\Delta$ divides over $\Gamma$.
\end{proposition}

\begin{proof}
Let $p$ be such a type. By definition, there must exist $\beta\in B(\Delta)$, and $(\phi, t)\in E^-_0$ such that $p(x)\models \phi(\nu(t(x_B)-\beta))$. Let us show that this formula divides over $\Gamma$. Let $Y$ be the coset $\nu(\beta)+\phi(C)$. Then we have $p(x)\models\nu((t(x_B)))+\phi(C)=Y$. As $p$ extends $p_0$, we have $Y\not\in \Gamma$. By \ref{GammaAlg}, we have $Y\not\in\acl^{eq}(\Gamma)$. Let $(Y_n)_{n<\omega}$ be a sequence of pairwise-distinct $\Gamma$-conjugates of $Y$. Then the definable sets $(\nu(t(x_B))+\phi(C)=Y_n)_n$ are $\Gamma$-conjugates and pairwise-disjoint. This implies that they divide over $\Gamma$, which concludes the proof.
\end{proof}

\begin{corollary}\label{orth2ACDelta}
    The space of $A$-types (resp. $C$-types) over $\Delta$ of realizations $u$ of $p_0$ such that $\indep{u}{\Gamma}{\Delta}{\d}$ is naturally homeomorphic to the direct image by $f_A^\Delta$ (resp. $f_C$) of the closed subspace of $S_0^{\Delta}$ of extensions of $p_0$ to $\Delta$ which do not divide over $\Gamma$. The same holds if we replace dividing by forking.
\end{corollary}

\begin{proof}
    As the spaces of types at hand are included in $S^\Delta$, we apply \cref{orthACDelta}.
\end{proof}

Now, one wonders what does the space of types which do not divide over $\Gamma$ look like. It turns out that it is a \textit{subrectangle} of $S^\Delta$:

\begin{theorem}\label{thmDivision}
Let $v, v'\in M$ be realizations of $p_0$ such that $\indep{v}{\Gamma}{\Delta}{\d}$ and $\indep{v'}{\Gamma}{\Delta}{\d}$. Let $w\in M$ be such that $f_A^\Delta(w)=f_A^\Delta(v)$, $f_C(w)=f_C(v')$. Then we have $\indep{w}{\Gamma}{\Delta}{\d}$.
\end{theorem}

\begin{proof}
Let $\phi(x)$ be a formula with parameters in $\Delta$ which is realized by $w$. By \cref{homeoRectangle}, \cref{stablePlongitude}, and the fact that $p_0$ and $p_\Delta$ are preimages by $f_C$, the partial type: $$p_0\cup
\left\lbrace f_A^\Delta(x)\models\tp(f_A^\Delta(w)/A^*(\Delta))\right\rbrace
\cup \left\lbrace f_C(x)\models\tp(f_C(w)/C^*(\Delta))\right\rbrace$$
generates the complete type $\tp(w/\Delta)$. By compactness, there must exist:
\begin{itemize}
\item A formula $\psi(x)$ with parameters in $\Gamma$ satisfied by all the realizations of $p_0$.
\item A formula $\phi_A(y)$ in $\mathcal{L}_A$, and with parameters in $A^*(\Delta)$, satisfied by $f_A^\Delta(w)$.
\item A formula $\phi_C(z)$ in $\mathcal{L}_C$, and with parameters in $C^*(\Delta)$, satisfied by $f_C(w)$.
\end{itemize}
such that:
$$M^{eq}\models\forall x\ \left[\left(
\psi(x)\wedge\phi_A\left(f_A^\Delta(x)\right)\wedge\phi_C(f_C(x))
\right)
\Longrightarrow \phi(x)
\right]$$
\par Note that only finitely many of the components of $f_A^\Delta$ appear in the formula $\phi_A(f_A^\Delta(x))$. Let $\delta$ be a finite tuple from $\Delta$ such that the formula $\phi_A\left(f_A^\Delta(x)\right)\wedge\phi_C(f_C(x))$ is $\delta$-definable, say, $\phi_A\left(f_A^\Delta(x)\right)=\theta_A(x, \delta)$ and $\phi_C(f_C(x))=\theta_C(x, \delta)$, with $\theta_A(x, y)$ and $\theta_C(x, z)$ parameter-free formulas. It suffices to show that the formula $\psi(x)\wedge\theta_A(x, \delta)\wedge\theta_C(x, \delta)$ does not divide over $\Gamma$.
\par Let $I=(\delta_n)_{n<\omega}$ be a $\Gamma$-indiscernible sequence containing $\delta$. As $\indep{v}{\Gamma}{\Delta}{\d}$, there must exist by \cref{indepDivisionSuitesIndisc} some $\overline{v}\equiv_\Delta v$ such that $I$ is $\Gamma\overline{v}$-indiscernible. Likewise, let $\overline{v'}\equiv_\Delta v'$ be such that $I$ is $\Gamma\overline{v'}$-indiscernible. Then $\overline{v}$ and $\overline{v'}$ are still realizations of $p_0$, thus there exists, by \cref{surjection} a tuple $\overline{w}\in M$ such that $f_A^\Delta(\overline{w})=f_A^\Delta(\overline{v})$ and $f_C(\overline{w})=f_C(\overline{v'})$ (in particular, $\overline{w}\models p_0$, thus $M^{eq}\models\psi(\overline{w})$). Let us show that $M\models\theta_A(\overline{w}, \delta_n)\wedge\theta_C(\overline{w}, \delta_n)$ for all $n<\omega$, which will conclude the proof. Choose $N<\omega$.
\par As $\overline{v'}\equiv_\Delta v'$, we have $M\models\theta_C(\overline{v'}, \delta)$. As $(\delta_n)_n$ is a $\Gamma\overline{v'}$-indiscernible sequence containing $\delta$ and $\delta_N$, we have $M\models \theta_C(\overline{v'}, \delta_N)$. As $f_C(\overline{w})=f_C(\overline{v'})$, we have $M\models \theta_C(\overline{w}, \delta_N)$.
\par Likewise, we have $M\models \theta_A(\overline{v}, \delta_N)$, but there is a subtlety to get around if we want to prove that $M\models\theta_A(\overline{w}, \delta_N)$: since $f_A^\Delta$ depends on $\Delta$, when replacing $\delta$ with its $\Gamma$-conjugate $\delta_N$, we also replace $f_A^\Delta$ with a $\Gamma$-conjugate\ldots  By strong homogeneity, let $\sigma\in\Aut(M^{eq}/\Gamma)$ be such that $\sigma(\delta)=\delta_N$. We have to show that $\sigma\left(f_A^\Delta\right)(\overline{v})=\sigma\left(f_A^\Delta\right)(\overline{w})$.
\par Since the definition of $f_A^\Delta$ is rather technical, we recall the various definitions:
\begin{itemize}
    \item $E^+_0=\left\lbrace (\phi, t)\in \Phi\times\mathcal{T}|\nu(t(u_B))+\phi(C)\in \faktor{C}{\phi(C)}(\Gamma)\right\rbrace$
    \item $E^+_\Delta=\left\lbrace (\phi, t)\in E^+_0|\exists \beta\in B(\Delta)\ \nu(t(u_B)-\beta)\in\phi(C)\right\rbrace$
    \item For each $(\phi, t)\in E^+_\Delta$, $\beta_{(\phi, t)}\in B(\Delta)$ is such that $\nu(t(u_B)-\beta_{(\phi, t)})$ lies in $\phi(C)$.
    \item $f_A^\Delta:\ (x_A, x_B, x_C)= (x_A, (\rho_\phi(t(x_B)-\beta_{(\phi, t)}))_{(\phi, t)\in E^+_\Delta})$.
\end{itemize}

\par Choose $(\theta, t)\in E^+_\Delta$, and let us show that:
$$\rho_\theta(t(v_B)-\sigma(\beta_{(\theta, t)}))=\rho_\theta(t(w_B)-\sigma(\beta_{(\theta, t)}))$$
\par By definition of $E^+_\Delta$ and $E^+_0$, the coset $\nu(\beta_{(\theta, t)})+\theta(C)$ belongs to $\Gamma$, thus $\nu(\sigma(\beta_{(\theta, t)}))$ belongs to that coset as well. It follows that the difference $\beta_{(\theta, t)}-\sigma(\beta_{(\theta, t)})$ is in $\theta(B)+\iota(A)$. As $\overline{v}$ and $\overline{w}$ both realize $p_0$, their image by $\nu$ belongs to that coset as well, and we conclude that:
$$
\begin{array}{ccl}
\rho_\theta(\overline{v}_B-\sigma(\beta_{(\theta, t)}))&=&\rho_\theta(\overline{v}_B-\beta_{(\theta, t)})+\rho_\theta(\beta_{(\theta, t)}-\sigma(\beta_{(\theta, t)}))\\
&=&\rho_\theta(\overline{w}_B-\beta_{(\theta, t)})+\rho_\theta(\beta_{(\theta, t)}-\sigma(\beta_{(\theta, t)}))\\
&=&\rho_\theta(\overline{w}_B-\sigma(\beta_{(\theta, t)}))
\end{array}
$$
\par The second equality follows from the fact that $f_A^\Delta(\overline{v})=f_A^\Delta(\overline{w})$, and the other equalities follow from the fact that $\rho_\theta$ is always applied here to elements of $\theta(B)+\iota(A)$, over which it restricts to a group homomorphism.
\end{proof}

\begin{corollary}\label{classifDivision}
Let $S_\d^\Delta$ be the space of every type in $S_0^\Delta$ which do not divide over $\Gamma$. Then the map $S_\d^\Delta\longrightarrow f_A^\Delta(S_\d^\Delta)\times f_C(S_\d^\Delta)$ is a homeomorphism.
\end{corollary}
\begin{proof}
    This map is clearly continuous. It is surjective by \cref{thmDivision}. By \cref{divisionPDelta}, we have $S_\d^\Delta\subset S^\Delta$, therefore this map is injective as a restriction of an injective map, by \cref{homeoRectangle}. We conclude by compactness and separation.
\end{proof}

\begin{corollary}\label{PSESDivisionVraiHomeo}
    Let $\mathcal{A}_\d^\Delta$ (resp. $\mathcal{C}_\d^\Delta$) be the space of $A$-types (resp. $C$-types) of realizations $u$ of $p_0$ such that $\indep{u}{\Gamma}{\Delta}{\d}$. Then $S_\d^\Delta$ is naturally homeomorphic to $\mathcal{A}_\d^\Delta\times\mathcal{C}_\d^\Delta$.
\end{corollary}
\begin{proof}
    This follows immediately from \cref{orth2ACDelta}.
\end{proof}

It follows that $\indep{u}{\Gamma}{\Delta}{\d}$ if and only if each of the partial types $\tp_C(u/\Delta)$ and $p_0\cup \tp_A(u/\Delta)$ have a realization $v$ such that $\indep{v}{\Gamma}{\Delta}{\d}$. Note that this statement is in general not equivalent to: $\tp_C(u/\Delta)$ and $p_0\cup\tp_A(u/\Delta)$ do not divide over $\Gamma$. This is related to the issue of forking being potentially stronger than dividing, or equivalently, that it is possible for a formula not to divide while containing only types that divide. More precisely, we should remind the reader that:
\begin{remark}
In any first-order theory, given a set of parameters $\Gamma$, the following are equivalent:
\begin{itemize}
\item For every set of parameters $\Delta$, for every partial type $p$ over $\Gamma\Delta$, $p$ does not divide over $\Gamma$ if and only if $p$ admits a realization $v$ such that $\indep{v}{\Gamma}{\Gamma\Delta}{\d}$.
\item For every set of parameters $\Delta$, for every $\Gamma\Delta$-definable set $X$, $X$ does not divide over $\Delta$ if and only if $X$ admits a realization $v$ such that $\indep{v}{\Gamma}{\Gamma\Delta}{\d}$.
\item Forking coincides with dividing over $\Gamma$, i.e. every definable set which forks over $\Gamma$ divides over $\Gamma$.
\end{itemize}
\end{remark}
\par A common example of a definable set which does not divide while all its realizations divide is in the well-known cyclical order $(\mathbb{Q}, \textup{cyc})$, with $\Gamma=\emptyset$, and $\Delta=\{*\}$ a singleton. In this setting, the empty partial type does not divide over $\Gamma$ (because it never does), but every type over $\Delta$ divides over $\Gamma$, thus all the realizations $v$ of the empty partial type satisfy $\notIndep{v}{\Gamma}{\Delta}{\d}$.
\par Nonetheless, it is true in any theory that a partial type $p$ over $\Gamma\Delta$ does not fork over $\Gamma$ if and only if there exists $v$ a realization of $p$ such that $\indep{v}{\Gamma}{\Delta}{\f}$. As a consequence, the non-trivial result that we can get from this section is:

\begin{corollary}\label{conditionNecessaireDivision}
If both partial types $p_0\cup \tp_A(u/\Delta)$ and $\tp_C(u/\Delta)$ do not fork over $\Gamma$, then we have $\indep{u}{\Gamma}{\Delta}{\d}$.
\end{corollary}
And the converse is trivial whenever forking coincides with dividing over $\Gamma$. In fact, the converse holds whenever $\indep{-}{\Gamma}{\Delta}{\f}=\indep{-}{\Gamma}{\Delta}{\d}$. In particular, it holds when $\Delta$ is a $|\Gamma|^+$-saturated model. Let us state it formally, as it will come in handy in the next section:
\begin{corollary}\label{deviationMonstre}
Suppose $\Delta$ is a $|\Gamma|^+$-saturated model. Then we have $\indep{u}{\Gamma}{\Delta}{\f}$ if and only if both partial types $p_0\cup \tp_A(u/\Delta)$ and $\tp_C(u/\Delta)$ do not fork over $\Gamma$.
\end{corollary}
\par Now, \cref{conditionNecessaireDivision} also suggests that \cref{deviationMonstre} holds for any $\Delta$. This is what we  prove in the next section.

\section{Forking}\label{PSESDeviation}
Recall that a partial type does not fork over $\Gamma$ if and only if it admits a global completion which does not divide over $\Gamma$. Note that the monster model $M^{eq}$ satisfies the hypothesis \ref{DeltaReel}, thus all the results that we obtained for $\Delta$ also hold for $M^{eq}$, by switching to some big elementary extension of $M$. 
By \cref{deviationExtGlobale}, forking for types over $\Delta$ involves dividing for types over $M^{eq}$, thus the results of the previous sections will help us to establish here a classification of the space of every type in $S_0^\Delta$ which do not fork over $\Gamma$.

\begin{definition}
Let $\mathcal{B}_0^\Delta$ be the Boolean algebra associated to $S_0^\Delta$. Let $e^\Delta$ be the embedding $\mathcal{B}_0^\Delta\longrightarrow\mathcal{B}_0^{M^{eq}}$. Define $F_A^\Delta$ (resp. $F_C^\Delta$) as the set of every element of $\mathcal{B}_0^{\Delta}$ which is equivalent to some formula of the form $\phi(f_A^{\Delta}(x))$ (resp. $\phi(f_C(x))$), with $\phi\in\mathcal{L}_A$ (resp. $\mathcal{L}_C$) with parameters in $\Delta$.
\end{definition}
There are several ways to formalize the Stone duality. We identify by convention an element of $\mathcal{B}_0^\Delta$ with a (clopen) subset of $S_0^\Delta$. With this formalism, given $I$ an ideal of $\mathcal{B}_0^\Delta$ (for instance: the ideal of definable sets $X$ such that $p_0\cup \{X\}$ forks over $\Gamma$), a type $p\in S_0^\Delta$ is inconsistent with $I$ (i.e. $p$ does not fork over $\Gamma$) if and only if $p\not\in \cup I$.
\par Recall that in a Boolean algebra $\mathcal{B}$, a subset is an ideal (as in ring theory) if and only if it is  non-empty, downward-closed and closed under finite join. In other words, the ideal generated by some set $P$ coincides with:
$$\left\lbrace X\in\mathcal{B}|\exists n<\omega\ \exists Y_1\ldots Y_n \in P\ X\leqslant (Y_1\vee\ldots \vee Y_n)\right\rbrace$$
moreover, if $P$ is already downward-closed, then this set coincides with:
$$\left\lbrace X\in\mathcal{B}|\exists n<\omega\ \exists Y_1\ldots Y_n \in P\ X= (Y_1\vee\ldots \vee Y_n)\right\rbrace$$
in particular, if $I$, $J$ are ideals of $\mathcal{B}$, then:
$$
I+J=\left\lbrace X\vee Y|X\in I,Y\in J\right\rbrace
$$
therefore in our setting:
$$
\cup(I+J)=(\cup I)\cup(\cup J)
$$
\begin{lemma}\label{lemmeAKEDeviation}
Let $I_A$ (resp. $I_C$) be an ideal of $\mathcal{B}_0^{M^{eq}}$ generated by some subset of $F_A^{M^{eq}}$ (resp. $F_C^{M^{eq}}$). Let $X_A\in F_A^\Delta$, and $X_C\in F_C^\Delta$. Then we have $e(X_A\wedge X_C)\in I_A+I_C$ if and only if $e(X_A)\in I_A$ or $e(X_C)\in I_C$.
\end{lemma}

\begin{proof}
The right-to-left direction is trivial.
\par Let us show the other direction by contraposition. Suppose that $e(X_A)\not\in I_A$ and $e(X_C)\not\in I_C$. Then there exists $p\in e(X_A)$, $q\in e(X_C)$ such that $p\not\in\cup I_A$, $q\not\in \cup I_C$. By \cref{surjection}, let $r\in S_0^{M^{eq}}$ be such that $f_A^{M^{eq}}(r)=f_A^{M^{eq}}(p)$, $f_C(r)=f_C(q)$. As $f_A^\Delta$ is a restriction of $f_A^{M^{eq}}$, we have $f_A^\Delta(r)=f_A^\Delta(p)$. As $p\in e(X_A)$, and $X_A\in F_A^\Delta$, we have $r\in e(X_A)$. Likewise, we have $r\in e(X_C)$, thus $r\in e(X_A\wedge X_C)$.
\par Now, it suffices to show that $r\not\in\cup (I_A+I_C)$. Note that $\cup(I_A+I_C)=(\cup I_A)\cup(\cup I_C)$, thus it suffices to show that $r\not\in \cup I_A$ and $r\not\in\cup I_C$.
\par By hypothesis on $I_A$, we have $\cup I_A=\cup(I_A\cap F_A^{M^{eq}})$. As $f_A^{M^{eq}}(r)=f_A^{M^{eq}}(p)$, for every $X\in F_A^{M^{eq}}$, we have $r\in X$ if and only if $p\in X$. As $p\not\in \cup I_A$, we have $r\not\in\cup(I_A\cap F_A^{M^{eq}})$, thus $r\not\in\cup I_A$.
\par Likewise, $r\not\in\cup I_C$, concluding the proof.
\end{proof}

Note that this kind of reasoning is very general, it may be used in more abstract settings which are not limited to expansions of PSES. The key property that we use to make this lemma work is that the Stone space at hand maps onto a rectangle by our two projection maps $f_A$, $f_C$. However, this lemma does not apply to every definable set, it only applies to \textit{actual} rectangles, i.e. definable sets of the form $X_A\wedge X_C$ as in the statement. This is where the rectangle $S^\Delta$ comes handy, because every element of the Boolean algebra corresponding to this Stone space may be written as a finite join of rectangles. Then, with \cref{divisionPDelta}, and the fact that $S^\Delta$ may be written as the intersection of a subset of $F_C^\Delta$, we have all the tools that we need to obtain our main result about forking. Let us formalize all this:

\begin{theorem}\label{thmDeviation}
Let $p_A$ (resp. $p_C$) be the partial type generated by all formulas in $F_A^{\Delta}$ (resp. $F_C^{\Delta}$) satisfied by $u$. Then we have $\indep{u}{\Gamma}{\Delta}{\f}$ if and only if both partial types $p_0\cup p_A$ and $p_C$ do not fork over $\Gamma$.
\end{theorem}

\begin{proof}
The left-to-right direction is trivial. Let us show the other direction, so suppose that $p_0\cup p_A$ and $p_C$ do not fork over $\Gamma$.
\par Let $I_A$ (resp. $I_C$) be the ideal of $\mathcal{B}_0^{M^{eq}}$ generated by formulas $\phi$ of $F_A^{M^{eq}}$ (resp. $F_C^{M^{eq}}$) such that $p_0\cup\{\phi\}$ (resp. $\phi$) forks over $\Gamma$. By \cref{deviationMonstre}, the (open) space of every type in $S_0^{M^{eq}}$ which forks over $\Gamma$ is exactly $\cup (I_A+I_C)$. It follows that the space of every type in $S_0^\Delta$ which forks over $\Gamma$ is $\cup e^{-1}(I_A+I_C)$.
\par If we had $u\not\models p_\Delta$, then no realization of $p_C$ would realize $p_\Delta$, therefore, by \ref{divisionPDelta}, we would have $\notIndep{v}{\Gamma}{\Delta}{\d}$ for every $v\models p_C$, contradicting our hypothesis that $p_C$ does not fork over $\Gamma$. So $u$ realizes $p_\Delta$.
\par Let $X$ be a $\Delta$-definable set containing $u$. As $u$ realizes $p_\Delta$, by \cref{homeoRectangle} and compactness, there exists $X_A\in F_A^\Delta$, $X_C\in F_C^\Delta$, and $Y$ a $\Delta$-definable set containing all the realizations of $p_\Delta$, such that, in $\mathcal{B}_0^\Delta$, we have $(Y\wedge X_A\wedge X_C)\leqslant X$. As $p_\Delta$ is a preimage by $f_C$, there must exist $Z\in F_C^\Delta$, containing $u$, and such that $Z\leqslant Y$. We have $(X_C\wedge Z)\in F_C^\Delta$, thus:
$$p_C\models x\in X_C\wedge x\in Z$$
\par As $p_C$ does not fork over $\Gamma$, we have $(X_C\wedge Z)\not\in e^{-1}(I_A+I_C)$, in particular $e(X_C\wedge Z)\not\in I_C$. Likewise, we have $e(X_A)\not\in I_A$, therefore by \cref{lemmeAKEDeviation} we have $e(X_A\wedge X_C\wedge Z)\not\in I_A+I_C$. In particular, this $\Delta$-definable set does not fork over $\Gamma$, and neither does $X$. This holds for all $X\in\mathcal{B}_0^\Delta$ containing $u$, so we conclude that $\indep{u}{\Gamma}{\Delta}{\f}$.
\end{proof}
\begin{corollary}\label{classifDeviation}
Let $S_\f^\Delta$ be the space of all types in $S_0^\Delta$ which do not fork over $\Gamma$. Then the map $S_\f^\Delta\longrightarrow f_A^\Delta(S_\f^\Delta)\times f_C(S_\f^\Delta)$ is a homeomorphism.
\end{corollary}

\begin{proof}
    Just as in \cref{classifDivision}, this map is a continuous injection between compact separated spaces. There remains to prove surjectivity.
    \par Let $(q_A(y),q_C(y))\in f_A^\Delta(S_\f^\Delta)\times f_C(S_\f^\Delta)$. By \cref{surjection}, let $q\in S_0^\Delta$ be such that $q(x)\models q_A(f_A^\Delta(x))$ and $q(x)\models q_C(f_C(x))$, and let $v$ realize $q$. As $(q_A(y),q_C(y))\in f_A^\Delta(S_\f^\Delta)\times f_C(S_\f^\Delta)$, each of the partial types $p_0\cup q_A(f_A^\Delta(x))$ and $q_C(f_C(x))$ admits a realization $w$ such that $\indep{w}{\Gamma}{\Delta}{\f}$. In particular, those partial types do not fork over $\Gamma$. By \cref{thmDeviation}, we have $\indep{v}{\Gamma}{\Delta}{\f}$, therefore $q$ is a preimage of $(q_A(y), q_C(y))$, concluding the proof.
\end{proof}

\begin{corollary}\label{PSESDeviationVraiHomeo}
    Let $\mathcal{A}_\f^\Delta$ (resp. $\mathcal{C}_\f^\Delta$) be the space of $A$-types (resp. $C$-types) of realizations $u$ of $p_0$ such that $\indep{u}{\Gamma}{\Delta}{\f}$. Then $S_\f^\Delta$ is naturally homeomorphic to $\mathcal{A}_\f^\Delta\times\mathcal{C}_\f^\Delta$.
\end{corollary}
\begin{proof}
    This follows immediately from \cref{orth2ACDelta}.
\end{proof}

\begin{corollary}\label{PSESThmPPal}
    We have $\indep{u}{\Gamma}{\Delta}{\f}$ if and only if both partial types $\tp_C(u/\Delta)$ and $p_0\cup\tp_A(u/\Delta)$ do not fork over $\Gamma$.
\end{corollary}

\section{A silly example}\label{PSESExemple}
In this section, we show a very basic example of a PSES of Abelian groups $M$ (without any expansion), with substructures $\Gamma\leqslant\Delta$ satisfying \ref{GammaAlg} and \ref{DeltaReel}, and $u\in M$ such that $\notIndep{u}{\Gamma}{\Delta}{\d}$, but both partial types $\tp_A(u/\Delta)$ and $\tp_C(u/\Delta)$ do not fork over $\Gamma$. This illustrates why the partial type $p_0$ is absolutely necessary in our main result \cref{thmDeviation}.
\par Let $M$ be the following $\Phi$-PSES of Abelian groups:

$$
\mathbb{Q}\times\{0\}\underset{\iota}{\longrightarrow}\mathbb{Q}\times\mathbb{Q}\underset{\nu}{\longrightarrow}\{0\}\times\mathbb{Q}
$$

\noindent with $\Phi=\{x=0\}$, which is indeed a fundamental system in torsion-free divisible Abelian groups. In fact, we should remind the reader that in the theory of such groups, any formula with one variable $x$ and parameters in some set $P$ is equivalent to a Boolean combination of formulas of the form $x-e=0$, with $e$ lying in the $\mathbb{Q}$-vector space generated by $P$.
\par Let $\Gamma=\dcl^{eq}(\{\nu(0, 1)\})$, $\Delta=\dcl^{eq}((1, 1))\supset\Gamma$, and let $u=(1, 1)$. Let $N$ be some $\aleph_1$-saturated, strongly $\aleph_1$-homogeneous elementary extension of $M$.
\par By quantifier elimination for torsion-free divisible Abelian groups, we have $C(\acl^{eq}(\Gamma))=\mathbb{Q}\cdot\nu(0, 1)=C(\Gamma)$, thus \ref{GammaAlg} holds, and \ref{DeltaReel} holds by definition.
\par One can check that $u\not\in\acl^{eq}(\Gamma)$, thus $\notIndep{u}{\Gamma}{\Delta}{\d}$.
\par The $C$-type of $u$ over $\Delta$ is isolated by a formula with parameters in $\Gamma$: $\nu(x)=\nu(0, 1)$. Let $v\in B(N)$ be such that $v\in \iota(A)$, but $v\neq w$ for every $w\in M$. By \cref{qe}, the type of $v+(0, 1)$ over $M$ is generated by the following partial type:

\[
\{\nu(x)=\nu(0, 1)\}
\]
\[
\cup \left\lbrace\rho_{x=0}(x-\beta)\neq 0|\beta\in B(M), \nu(\beta)=\nu(0, 1)\right\rbrace
\]
\[\cup 
\left\lbrace\rho_{x=0}(x-\beta)= 0|\beta\in B(M), \nu(\beta)\neq\nu(0, 1)\right\rbrace
\]

As this partial type is $\Aut(M^{eq}/\Gamma)$-invariant, $\tp(v+(0, 1)/M)$ is also $\Aut(M^{eq}/\Gamma)$-invariant. In particular, this type does not fork over $\Gamma$. As $v+(0, 1)$ realizes $\tp_C(u/\Delta)$, we can deduce that $\tp_C(u/\Delta)$ does not fork over $\Gamma$.
\par Using \cref{qe}, one may show that for all $\lambda,\mu\in\mathbb{Q}$, we have $(\lambda,\mu)\equiv_\Gamma(2\lambda, 2\mu)$, thus $B(\Gamma)=\{\iota(0, 0)\}$. In particular, as: $$\rho_{x=0}(\lambda\cdot u-\iota(0, 0))=\rho_{x=0}(\lambda, \lambda)=(0, 0)$$ for every $\lambda\in\mathbb{Q}$, we have $\tp_A(u/\Delta)=\tp_A(\iota(0, 0)/\Delta)$. As $\iota(0, 0)\in\Gamma$, we have of course $\indep{\iota(0, 0)}{\Gamma}{\Delta}{\f}$, thus $\tp_A(u/\Delta)$ does not fork over $\Gamma$, which concludes the example.

\section{Models}\label{sectPSESModeles}

In \cref{thmDeviation}, we reduce the problem of characterizing forking in general to the problem of characterizing forking for certain partial types which encode data related to $A$ and $C$. However, the partial types at hand are still types in the full PSES, it would be better if we could reduce the problem to checking forking in the reducts $A^*$ and $C^*$. While this is probably not possible in full generality, we isolate conditions on $\Gamma$ which make it possible. Those conditions hold in particular when $\Gamma$ is a model.

\begin{definition}\label{defPSESModeles}
    In this section, we assume that for every $\phi\in\Phi$, for every coset $X\in\faktor{C}{\phi(C)}(\Gamma)$, there exists $\beta\in B(\Gamma)$ such that $\nu(\beta)\in X$. We also assume that $\Gamma$ is generated by reals, i.e. $\Gamma$ coincides with $\dcl^{eq}(A^*(\Gamma)\cup B(\Gamma)\cup C^*(\Gamma))$.
\end{definition}

\begin{remark}
    With that assumption, we have $E^+_\Delta=E^+_0$, and the $\beta_{(\phi, t)}$ can be (and are, from now on) chosen in $B(\Gamma)$ for any $(\phi, t)\in E^+_\Delta$. In particular, $f_A^\Delta$ is $\Gamma$-$*$-definable and does not depend on $\Delta$, so we might as well call it $f_A$ in the rest of this section.
\end{remark}

\begin{lemma}\label{PSESModeleDcleq}
    We have $A^*(\dcl^{eq}(\Gamma u))=A^*(\dcl^{eq}(A^*(\Gamma)f_A(u)))$, as well as $C^*(\dcl^{eq}(\Gamma u))=C^*(\dcl^{eq}(C^*(\Gamma)f_C(u)))$.
\end{lemma}

\begin{proof}
    The part of the statement about $C$ follows immediately from \cref{stablePlongitude}, and the fact that $\Gamma$ is generated by reals.
    \par As for the part about $A$, by \cref{qe}, it suffices to show that $\rho_\phi(t(u_B)-\beta)\in A^*(\dcl^{eq}(A^*(\Gamma)f_A(u)))$ for every, $\phi\in\Phi$, $\beta\in B(\Gamma)$, and every term $t$. We follow the same case disjunction as in the proof of \cref{homeoRectangle}:

    \begin{itemize}
        \item If $(\phi, t)\in E^+_0$, and $\beta-\beta_{(\phi, t)}\in\phi(B)+\iota(A)$, then:
        $$\rho_\phi(t(u_B)-\beta)=\rho_\phi(t(u_B)-\beta_{(\phi, t)})+\rho_\phi(\beta_{(\phi, t)}-\beta)$$
        We conclude as $\rho_\phi(t(u_B)-\beta_{(\phi, t)})$ is a projection of $f_A(u)$, and as $\rho_\phi(\beta_{(\phi, t)}-\beta)$ is in $A^*(\Gamma)$.
        \item If $(\phi, t)\in E^+_0$, and $\beta-\beta_{(\phi, t)}\not\in\phi(B)+\iota(A)$, then $\rho_\phi(t(u_B)-\beta)=0$.
        \item If $(\phi, t)\in E^-_0$, then $\rho_\phi(t(u_B)-\beta)=0$.
    \end{itemize}
    this concludes the proof.
\end{proof}

\begin{theorem}\label{PSESModelesThmDivision}
    The following are equivalent:

    \begin{itemize}
        \item $\indep{u}{\Gamma}{\Delta}{\d}$.
        \item $\indep{f_A(u)}{A^*(\Gamma)}{A^*(\Delta)}{\d}$ in the reduct $A^*(M)$, and $\indep{f_C(u)}{C^*(\Gamma)}{C^*(\Delta)}{\d}$ in the reduct $C^*(M)$.
    \end{itemize}
\end{theorem}

\begin{proof}
    The proof may be seen as an easier version than that of \cref{thmDivision}.
    \par Suppose $\notIndep{f_A(u)}{A^*(\Gamma)}{A^*(\Delta)}{\d}$ in $A^*(M)$. Let $\phi_A(y, z)\in\mathcal{L}_A$, and let $\delta\in A^*(\Delta)$ be such that $A^*(M)\models \phi_A(f_A(u), \delta)$, and $\phi_A(y, \delta)$ divides over $A^*(\Gamma)$. Let $(\delta_i)_i$ be a $A^*(\Gamma)$-indiscernible sequence witnessing division. By \cref{stablePlongitude}, and as $\Gamma$ is generated by reals, $(\delta_i)_i$ is actually $\Gamma$-indiscernible. It follows that $\phi_A(y, \delta)$ divides over $\Gamma$, therefore $\phi_A(f_A(x), \delta)$ divides over $\Gamma$ by \cref{divisionPreimage}, and $\notIndep{u}{\Gamma}{\Delta}{\d}$.
    \par Likewise, if $\notIndep{f_C(u)}{C^*(\Gamma)}{C^*(\Delta)}{\d}$ in $C^*(M)$, then $\notIndep{u}{\Gamma}{\Delta}{\d}$, therefore the first condition of the statement implies the second one.
    \par Conversely, suppose by contradiction that the first condition fails, and the second one holds. As $\indep{f_C(u)}{C^*(\Gamma)}{C^*(\Delta)}{\d}$, $u$ realizes $p_\Delta$ by the proof of \cref{divisionPDelta}. By \cref{homeoRectangle}, there must exist a $\Delta$-definable set $X$ implied by $p_\Delta$, $\alpha$ a tuple from $A^*(\Delta)$, $\gamma$ a tuple from $C^*(\Delta)$ and formulas $\phi_A(y, \alpha)\in\mathcal{L}_A, \phi_C(z, \gamma)\in\mathcal{L}_C$ such that: $$M\models \phi_A(f_A(u), \alpha)\wedge \phi_C(f_C(u), \gamma)$$ and the formula $x\in X\wedge \phi_A(f_A(x), \alpha)\wedge\phi_C(f_C(x), \gamma)$ divides over $\Gamma$. As $\tp_C(u/\Delta)$ implies $p_\Delta$, we may assume $X$ is implied by $\phi_C(f_C(x), \gamma)$, therefore the formula $\phi_A(f_A(x), \alpha)\wedge\phi_C(f_C(x), \gamma)$ divides over $\Gamma$. Let $(\alpha_i\gamma_i)_i$ be a $\Gamma$-indiscernible sequence containing $\alpha\gamma$, and let us show that it does not witness division for this definable set, yielding a contradiction. By the second condition, and \cref{indepDivisionSuitesIndisc}, let $(\alpha'_i)_i\equiv_{A^*(\Delta)}(\alpha_i)_i$, $(\gamma'_i)_i\equiv_{C^*(\Delta)}(\gamma_i)_i$ be such that $(\alpha'_i)_i$ is $A^*(\Gamma)f_A(u)$-indiscernible, and $(\gamma'_i)_i$ is $C^*(\Gamma)f_C(u)$-indiscernible. By \cref{PSESModeleDcleq}, and \cref{stablePlongitude} applied to $\dcl^{eq}(\Gamma u)$, the sequence $(\alpha'_i\gamma'_i)_i$ is in fact $\Gamma u$-indiscernible. However, as there exists $j$ such that $\alpha_j\gamma_j=\alpha\gamma$, we also have $\alpha'_j\gamma'_j=\alpha\gamma$. In particular, we have $M\models \phi_A(f_A(u), \alpha'_i)\wedge \phi_C(f_C(u), \gamma'_i)$ for every $i$ by indiscernibility. It follows that $$\left\lbrace \phi_A(f_A(x), \alpha'_i)\wedge \phi_C(f_C(x), \gamma'_i)|i \right\rbrace$$ is consistent, hence the same can be said of $$\left\lbrace \phi_A(f_A(x), \alpha_i)\wedge \phi_C(f_C(x), \gamma_i)|i \right\rbrace$$ concluding the proof.
\end{proof}

\begin{corollary}\label{PSESModelesDivisionFinal}
    Let $U$ be the substructure of $M$ generated by $u$ and by $A^*(\Gamma)\cup B(\Gamma)\cup C^*(\Gamma)$. Then we have $\indep{U}{\Gamma}{\Delta}{\d}$ if and only if we have $\indep{A^*(U)}{A^*(\Gamma)}{A^*(\Delta)}{\d}$ and $\indep{C^*(U)}{C^*(\Gamma)}{C^*(\Delta)}{\d}$ in the respective reducts $A^*(M)$ and $C^*(M)$.
\end{corollary}

\begin{proof}
    This follows from the fact that $f_A$ and $f_C$ are defined using terms from the language of $M$ with parameters in $\Gamma$.
\end{proof}

\begin{theorem}
    The following are equivalent:

    \begin{itemize}
        \item $\indep{u}{\Gamma}{\Delta}{\f}$.
        \item $\indep{f_A(u)}{A^*(\Gamma)}{A^*(\Delta)}{\f}$ in the reduct $A^*(M)$, and $\indep{f_C(u)}{C^*(\Gamma)}{C^*(\Delta)}{\f}$ in the reduct $C^*(M)$.
    \end{itemize}
\end{theorem}

\begin{proof}
    Let us use the same notations as in \cref{PSESDeviation}.
    \par By \cref{PSESModelesThmDivision}, the ideal of $\mathcal{B}_0^{M^{eq}}$ generated by definable sets which fork over $\Gamma$ coincides with $I_A+I_C$, where:
    \begin{itemize}
        \item $I_A$ is the ideal generated by formulas in $F_A^{M^{eq}}$ of the form $\phi_A(f_A(x), \alpha)$, where $\phi_A(y, \alpha)$ divides over $A^*(\Gamma)$ in the reduct $A^*(M)$.
        \item Likewise for $I_C$, by replacing $A^*$, $f_A$ and $F_A^{M^{eq}}$ by $C^*$, $f_C$ and $F_C^{M^{eq}}$.
    \end{itemize}
    Note that, without the hypothesis of this section, those ideals are in general smaller than the ones defined in the proof of \cref{thmDeviation}.
    \par Suppose $\notIndep{f_A(u)}{A^*(\Gamma)}{A^*(\Delta)}{\f}$ in $A^*(M)$. Then there are finitely many formulas $(\phi_i(f_A(x), \alpha))_i$ in $I_A$ such that $\tp(f_A(u)/A^*(\Delta))\models\bigvee\limits_i \phi_i(y, \alpha)$. It follows that $\tp(u/\Delta)\in \cup e^{-1}(I_A)$, therefore we have $\notIndep{u}{\Gamma}{\Delta}{\f}$.
    \par Likewise, if $\notIndep{f_C(u)}{C^*(\Gamma)}{C^*(\Delta)}{\f}$, then $\notIndep{u}{\Gamma}{\Delta}{\f}$, which proves the top-to-bottom direction of the statement.
    \par Conversely, suppose $\notIndep{u}{\Gamma}{\Delta}{\f}$. We may assume $u$ realizes $p_\Delta$, for otherwise $\notIndep{f_C(u)}{C^*(\Gamma)}{C^*(\Delta)}{\d}$. As in \cref{PSESModelesThmDivision}, by \cref{homeoRectangle}, and the fact that $\tp_C(u/\Delta)$ implies $p_\Delta$, there exists $\Delta$-definable sets $X_A\in F_A^\Delta$, $X_C\in F_C^\Delta$ such that $X_A\wedge X_C$ forks over $\Gamma$, and $u\in X_A\wedge X_C$. By \cref{lemmeAKEDeviation}, either $X_A\in I_A$, in which case $\notIndep{f_A(u)}{A^*(\Gamma)}{A^*(\Delta)}{\f}$ in $A^*(M)$, or $X_C\in I_C$, in which case $\notIndep{f_C(u)}{C^*(\Gamma)}{C^*(\Delta)}{\f}$ in $C^*(M)$. This concludes the proof.
\end{proof}

\begin{corollary}\label{PSESModelesDeviation}
    Let $U$ be the substructure of $M$ generated by $u$ and by $A^*(\Gamma)\cup B(\Gamma)\cup C^*(\Gamma)$. Then we have $\indep{U}{\Gamma}{\Delta}{\f}$ if and only if we have $\indep{A^*(U)}{A^*(\Gamma)}{A^*(\Delta)}{\f}$ and $\indep{C^*(U)}{C^*(\Gamma)}{C^*(\Delta)}{\f}$ in the respective reducts $A^*(M)$ and $C^*(M)$.
\end{corollary}

\bibliographystyle{plain}
\bibliography{biblio}

\end{document}